\documentclass[10pt]{amsart}
\usepackage{amsfonts,amssymb,amscd,amsmath,enumerate,verbatim,calc}
\headsep=1cm \numberwithin{equation}{section}
\newtheorem{thm}{Theorem}[section]
\newtheorem{cor}[thm]{Corollary}
\newtheorem{lem}[thm]{Lemma}

\newcommand{\Ann}{\mbox{Ann}\,}

\newcommand{\Hom}{\mbox{Hom}\,}
\newcommand{\Ext}{\mbox{Ext}\,}
\newcommand{\Tor}{\mbox{Tor}\,}

\newcommand{\Tot}{\mbox{Tot}\,}

\renewcommand{\dim}{\mbox{dim}\,}

\newcommand{\cd}{\mbox{cd}\,}
\newcommand{\q}{\mbox{q}\,}

\newcommand{\im}{\mbox{Im}\,}

\renewcommand{\H}{\mbox{H}}

\newcommand{\lo}{\longrightarrow}

\newcommand{\fa}{\mathfrak{a}}
\newcommand{\fb}{\mathfrak{b}}

\newcommand{\fm}{\mathfrak{m}}

\newcommand{\C}{C}

\begin{document}

\title[Torsion functors of local cohomology modules]
 {Torsion functors of local cohomology modules}

\bibliographystyle{amsplain}

     \author[M. T. Dibaei]{Mohammad T. Dibaei$^{1}$}
     \author[A. Vahidi]{Alireza Vahidi$^{2}$}

\address{$^{1}$ Faculty of Mathematical Sciences and Computer, Tarbiat Moallem University, Tehran, Iran; and
School of Mathematics, Institute for Research in Fundamental
Sciences (IPM), P.O. Box: 19395-5746, Tehran, Iran.}
\email{dibaeimt@ipm.ir}

\address{$^{2}$ Payame Noor University (PNU), Iran.}
\email{vahidi.ar@gmail.com}

\keywords{coatomic modules, local cohomology modules, minimax modules, spectral sequences.\\
The research of the first author was in part supported by a grant
from IPM (No. 88130126).}



\begin{abstract}
Through a study of torsion functors of local cohomology modules
we improve some non-finiteness results on the top non-zero local
cohomology modules with respect to an ideal.
\end{abstract}

\maketitle

\section{Introduction}
Let $R$ be a commutative Noetherian ring with non-zero identity. We use symbols
$\fa$, $M$, and $X$ as an ideal of $R$, a finite (i.e. finitely generated)
$R$--module, and an arbitrary $R$--module which is not necessarily finite.
The $i$th local cohomology module of $X$ with respect to $\fa$ is denoted
by $\H_\fa^i(X)$.

For all $i\geq 0$, it is well known that $\H^i_\fm(M)$ is Artinian for any
maximal ideal $\fm$ of $R$. In particular, $\Hom_R(R/\fm, \H^i_\fm(M))$ is finite.
Grothendieck asked, in \cite {G}, whether a similar statement is valid if
$\fm$ is replaced by an arbitrary ideal of $R$. Hartshorne gave a counterexample
in \cite {Ha2} and raised the question whether $\Ext^{i}_{R}(R/\fa, \H^{j}_{\fa}(M))$
is finite for all $i$ and $j$, and proved this is the case when $R$ is a
complete regular local ring and $\dim(R/\fa)= 1$. This result was later extended to
more general rings by Delfino and Marley  (\cite [Theorem 1]{DM}).

For an $R$--module $X$, Melkersson \cite [Theorem 2.1]{Mel} proved that
$\Ext_R^i(R/\fa, X)$ is finite for all $i$ if and only if
$\Tor_i^R(R/\fa, X)$ is finite for all $i$. Summarizing the above results,
we see that for any ideal $\fa$ of $R$ with $\dim(R/\fa)\leq 1$,
$\Tor_i^R(R/\fa, \H_\fa^j(M))$ is finite for all $i$ and $j$.
This result inspired us to study $\Tor_i^R(R/\fa, \H_\fa^j(X))$ in general for an
arbitrary $R$--module $X$.
Note that there are some attempts to study $\Tor_0^R(R/\fa, \H_\fa^j(X))$ in \cite {AT} and
$\Tor_i^R(R/\fa, \H_\fa^j(M))$ in \cite {KH}.

In Section 2, we present some technical results (Lemma \ref {2-1} and Theorem \ref {2-2}) which
show that, in certain situation, the torsion module $\Tor_i^R(R/\fa, \H_\fa^j(X))$
is in a Serre subcategory of the category of $R$--modules. Recall that $\mathcal{S}$ is a Serre
subcategory of the category of $R$--modules if for any exact sequence
\begin{equation}\label{1--1}
0\longrightarrow X'\longrightarrow X\longrightarrow X''\longrightarrow 0
\end{equation}
the module $X$ is in $\mathcal{S}$ if and only if $X'$ and $X''$ are in $\mathcal{S}$.
Always, $\mathcal{S}$ stands for a Serre subcategory of the category of $R$--modules.

Section 3 consists of applications. In Corollary \ref {3-3}, we show that, for certain integer $i$, $\H_\fa^i(X)$ may not be finite, coatomic, or minimax. Recall that, an $R$--module $X$ is said to be {\it coatomic} (resp. {\it minimax}) if any submodule of $X$ is contained in a maximal submodule of $X$ (resp. if there is a finite submodule $X'$ of $X$ such that $X/X'$ is Artinian). Finally, we show that, for a positive integer $n$, the statement ``$\H_\fa^i(X)$ is coatomic for all $i\geq n$'' is equivalent to each of the statements ``$\H_\fa^i(X)$ is finite for all $i\geq n$'' and ``$\H_\fa^i(X)= 0$ for all $i\geq n$''; also the statement ``$\H_\fa^i(X)$ is minimax for all $i\geq n$''  is equivalent to the statement ``$\H_\fa^i(X)$ is Artinian for all $i\geq n$'' (Corollaries \ref {3-4} and \ref {3-5}).\\

\section{Main result}
In this section, $c$ denotes the arithmetic rank of the ideal $\fa$,
so that there exist elements $x_1, \cdots, x_c$ of $R$ such that
$\sqrt{\fa}= (x_1, \cdots, x_c)$, also $\C(X)^\bullet$ denotes the
\v{C}ech complex of $X$ with respect to $x_1, \cdots, x_c$. It is
well known that the $i$th cohomology module of $\C(X)^\bullet$ is
isomorphic to the $i$th local cohomology module $\H_\fa^i(X)$ (see
\cite [Theorem 5.1.19] {BS}).



Our method is based on the following lemma. We adopt the notation as in \cite {Rot}.

\begin{lem} \label {2-1} Assume that $X$ and $N$ are $R$--modules such that $N$ is
$\fa$-torsion. Then there is a first quadrant spectral sequence
\begin{equation} \label{2--1}
E^{2}_{p, q}:= \emph{\Tor}^{R}_{p}(N, \emph{\H}^{c-q}_{\fa}(X))
\underset{p}\Longrightarrow \emph{\Tor}^{R}_{p+q-c}(N, X).
\end{equation}
\end{lem}

\begin{proof} Let $F_{\bullet}$ be a free resolution of $N$
and consider the first quadrant bicomplex $\mathcal{T}=
\{\C(F_{p}\otimes_{R} X)^{c-q}\}$. We denote the total complex of $\mathcal{T}$ by $\Tot(\mathcal{T})$. The first filtration has $E^{2}$
term the iterated homology $H'_{p}H''_{p, q}(\mathcal{T})$. By \cite
[Theorem 5.1.19] {BS}, we have
$$H''_{p, q}(\mathcal{T})= H^{c-q}(\C(F_{p}\otimes_{R} X)^{\bullet})= H^{c-q}_{\fa}(F_{p}\otimes_{R} X)= F_{p}\otimes_{R} H^{c-q}_{\fa}(X).$$
Hence
$$^{I}\!E^{2}_{p, q}= H_{p}(F_{\bullet}\otimes_{R} \H^{c-q}_{\fa}(X))= \Tor^{R}_{p}(N, H^{c-q}_{\fa}(X)).$$

On the other hand, the second filtration has
$E^{2}$ term the iterated homology $H''_{p}H'_{q, p}(\mathcal{T})$.
We have
$$H'_{q, p}(\mathcal{T})= H_{q}(\C(R)^{c-p}\otimes_{R}F_\bullet\otimes_R X)= \C(R)^{c-p}\otimes_{R} H_{q}(F_\bullet\otimes_R X)= \C(\Tor^{R}_{q}(N, X))^{c-p}.$$
Thus, again by \cite [Theorem 5.1.19] {BS},
$$^{II}\!E^{2}_{p, q}= H^{c-p}(\C(\Tor^{R}_{q}(N, X))^{\bullet})= H^{c-p}_{\fa}(\Tor^{R}_{q}(N, X)).$$
Since $\Tor^{R}_{q}(N, X)$ is $\fa$--torsion for all $q$,
$$^{II}\!E^{2}_{p, q}\cong \left\lbrace
           \begin{array}{c l}
              \Tor^{R}_{q}(N, X)\ \   & \text{if  \ \ $p= c$,}\\
              0\ \   & \text{if  \ \ $p\neq c$.}
           \end{array}
        \right.$$
Therefore this spectral sequence collapses at the $c$th column and
so
$$H_{p+q}(\Tot (\mathcal{T}))=\ ^{II}\!E^{2}_{c, p+q-c}= \Tor^{R}_{p+q-c}(N, X)$$
which yields the assertion.
\end{proof}



It is our main object to find out when a torsion functor of a local cohomology module
is in a Serre subcategory $\mathcal{S}$. Note that the following subcategories are examples of
Serre subcategories of the category of $R$--modules:
finite $R$--modules; Artinian $R$--modules; coatomic $R$--modules (\cite {Z2}); minimax $R$--modules (\cite {Z}); and
trivially the zero $R$--module. In the following theorem, we find some sufficient conditions for
this purpose.

\begin{thm} \label {2-2} Suppose that $X$ and $N$ are $R$--modules such that $N$ is $\fa$--torsion.
Assume also that $s, t$ are non-negative integers such that
               \begin{itemize}
                   \item[(i)]{$\emph{\Tor}^{R}_{s- t}(N, X)$ is in $\mathcal{S}$,}
                   \item[(ii)]{$\emph{\Tor}^{R}_{s- t+ i- 1}(N, \emph{\H}^{i}_{\fa}(X))$ is in
                   $\mathcal{S}$ for all $i$, $0\leq i\leq t- 1,$ and}
                   \item[(iii)]{$\emph{\Tor}^{R}_{s- t+ i +1}(N, \emph{\H}^{i}_{\fa}(X))$ is in
                   $\mathcal{S}$ for all $i$, $t+ 1\leq i\leq c.$}
               \end{itemize}
Then $\emph{\Tor}^{R}_{s}(N, \emph{\H}^{t}_{\fa}(X))$ is in
$\mathcal{S}$.
\end{thm}

\begin{proof} We may assume that $t\leq c$. Set $u= c- t$, $n= s+u$, and consider the spectral
sequence (\ref {2--1}). For all $r\geq 2$, let $Z_{s,
u}^{r}= \ker(E_{s, u}^{r} \longrightarrow E_{s- r, u+ r -1}^{r})$
and $B_{s, u}^{r}= \im(E_{s+ r, u- r+1}^{r}\longrightarrow E_{s,
u}^{r})$. So that we have the exact sequences:
$$0\longrightarrow Z_{s, u}^{r}\longrightarrow E_{s, u}^{r}\longrightarrow E_{s, u}^{r}/Z_{s, u}^{r}
\longrightarrow 0$$ and
$$0\longrightarrow B_{s, u}^{r}\longrightarrow Z_{s, u}^{r}\longrightarrow E_{s, u}^{r+ 1}\longrightarrow 0.$$
Note that $E_{s- r, u+ r -1}^{2}$ and $E_{s+ r, u- r+1}^{2}$ are in
$\mathcal{S}$ by assumptions (ii) and (iii), so that their
subquotients $E_{s- r, u+ r -1}^{r}$ and $E_{s+ r, u- r+1}^{r}$,
respectively, are also in $\mathcal{S}$. Thus $E_{s, u}^{r}/Z_{s,
u}^{r}$ and $B_{s, u}^{r}$ are in $\mathcal{S}$. It follows by the
above exact sequences that if $E_{s, u}^{r+ 1}$ is in $\mathcal{S}$,
then $E_{s, u}^{r}$ is in $\mathcal{S}$.

As we have $E_{s+ r, u- r+1}^{r}= 0= E_{s- r, u+ r- 1}^{r}$ for all
$r \geq s+ u+ 2$, we obtain $E_{s, u}^{\infty}= E_{s, u}^{s +u+ 2}$.
To complete the proof, it is enough to show that $E_{s, u}^{\infty}$
is in $ \mathcal{S}$.

There exists a finite filtration
$$0= \phi^{-1}H_{n}\subseteq \phi^{0}H_{n}\subseteq \cdots\subseteq \phi^{n- 1}H_{n}\subseteq
\phi^{n}H_{n}= \Tor^{R}_{s- t}(N, X)$$ such that $E_{r,
n-r}^{\infty}= \phi^{r}H_{n}/\phi^{r-1}H_{n}$ for all $r$, $0\leq
r\leq n$. Since $\Tor^{R}_{s- t}(N, X)$ is in $\mathcal{S}$,
$\phi^{s}H_{n}$ is also in $\mathcal{S}$. Thus $E_{s, u}^{\infty}=
\phi^{s}H_{n}/\phi^{s- 1}H_{n}$ is in $ \mathcal{S}$ as we desired.
\end{proof}

\section{Applications}



One can use Theorem \ref {2-2} to study some sufficient conditions
for finiteness of torsion functors of local cohomology modules. This
is the subject of \cite [Theorem 4.1] {KH} which shows that, for
given integers $s, t$ and given ideals $\fa\subseteq \fb$,
$\Tor^{R}_{s}(R/\fb, \H^{t}_{\fa}(M))$ is finite whenever $M$ is a
finite $R$--module with $\dim_R(M)< \infty$, $\Tor^{R}_{s- t+ i-
1}(R/\fb, \H^{i}_{\fa}(M))$ is finite for all $i< t$, and
$\Tor^{R}_{s- t+ i +1}(R/\fb, \H^{i}_{\fa}(M))$ is finite for all
$i> t$. In the following, we prove this theorem without assuming
that $M$ is finite and with no restrictions on dimension of $M$.

\begin{cor} \label {3-1} \emph{(cf.} \cite [Theorem 4.1]{KH}\emph{)}
Suppose that $X$ and $N$ are $R$--modules such that $N$ is
$\fa$--torsion. Assume also that $s, t$ are non-negative integers
such that
               \begin{itemize}
                   \item[(i)]{$\emph{\Tor}^{R}_{s- t}(N, X)$ is finite,}
                   \item[(ii)]{$\emph{\Tor}^{R}_{s- t+ i- 1}(N, \emph{\H}^{i}_{\fa}(X))$ is finite for all $i$, $0\leq i\leq t- 1,$ and}
                   \item[(iii)]{$\emph{\Tor}^{R}_{s- t+ i +1}(N, \emph{\H}^{i}_{\fa}(X))$ is finite for all $i$, $t+ 1\leq i\leq c.$}
               \end{itemize}
Then $\emph{\Tor}^{R}_{s}(N, \emph{\H}^{t}_{\fa}(X))$ is finite.
\end{cor}

\begin{proof}
In Theorem \ref{2-2}, take $\mathcal{S}$ to be the subcategory of
finite $R$--modules. The result follows.
\end{proof}



Let $n$ be a positive integer and $\H_\fa^i(X)$ is in $\mathcal{S}$ for all $i> n$.
In \cite [Theorem 3.1] {AT}, it is shown that $\H^{n}_{\fa}(X)/\fa \H^{n}_{\fa}(X)$
is in $\mathcal{S}$ whenever $X$ is a weakly Laskerian $R$--module (i.e. the set of
associated primes of any quotient module of $X$ is finite) and $X$ has finite Krull
dimension. In the first part of the following result, we generalize the statement by
removing all conditions on $X$.

\begin{cor} \label {3-2} Let $X$ be an $R$--module and let $\mathcal{S}$ be a Serre subcategory of the
category of $R$--modules such that, for a given integer $n$,
$\emph{\H}^{i}_{\fa}(X)$ is in $\mathcal{S}$ for all $i> n$. Assume
that $N$ is an $\fa$--torsion finite $R$--module and that $\fb$ is
an ideal of $R$ with $\fa\subseteq \sqrt{\fb}$. Then the following
statements hold true.
               \begin{itemize}
                   \item[(i)]{If $n> 0$, then $N\otimes_{R} \emph{\H}^{n}_{\fa}(X)$ is in $\mathcal{S}$.
                   In particular, $\emph{\H}^{n}_{\fa}(X)/\fb\emph{\H}^{n}_{\fa}(X)$
                   is in $\mathcal{S}.$}
                   \item[(ii)]{If $n> 1$, then $\emph{\Tor}^{R}_{1}(N, \emph{\H}^{n}_{\fa}(X))$ is in $\mathcal{S}$.
                   In particular, $\emph{\Tor}^{R}_{1}(R/\fb, \emph{\H}^{n}_{\fa}(X))$
                   is in $\mathcal{S}.$}
               \end{itemize}
\end{cor}

\begin{proof} Put $t= n$  in Theorem \ref{2-2}. For the first part take $ s= 0$; and, for the second part, take $s= 1$.
\end{proof}



In the course of the remaining parts of the paper by
$\cd_{\mathcal{S}}(\fa, X)$ ($\mathcal{S}$--cohomological dimension
of $X$ with respect to $\fa$) we mean the largest integer $i$ in
which $\H^{i}_{\fa}(X)$ is not in $\mathcal{S}$ (see \cite
[Definition 3.4] {AT} or  \cite [Definition 3.5] {AM}). If
$\mathcal{S}=0$, then $\cd_{\mathcal{S}}(\fa, X)= \cd(\fa, X)$ as in
\cite{Ha1}. When $\mathcal{S}$ is the category of Artinian
$R$--modules, we write $\q_{\fa}(X):= \cd_{\mathcal{S}}(\fa, X)$.
Note that  $\q_{\fa}(X)= \q(\fa, X)$ if $R$ is local as in
\cite[Definition 3.1]{DY1}.

As an application of Corollary \ref {3-2}, we bring the following result
which is essentially about non-finiteness of $\H^{\tiny\cd_{\mathcal{S}}(\fa, X)}_{\fa}(X)$
where $X$ is an arbitrary $R$--module. In \cite [Theorem 3.2] {H}, it is shown that
$\H^{\tiny\cd(\fa, X)}_{\fa}(X)$ is not coatomic whenever $0< \cd(\fa, X)= \cd(\fa, R/\Ann(X))$.
In the second part of the following result, the equality condition is removed.

\begin{cor} \label {3-3} For an arbitrary $R$--module $X$, the following statements hold true.
               \begin{itemize}
                   \item[(i)]{If $\emph{\cd}_{\mathcal{S}}(\fa, X)> 0$, then $\emph{\H}^{\emph{\cd}_{\mathcal{S}}
                   (\fa, X)}_{\fa}(X)/T$ is not finite for any submodule $T$ of $\emph{\H}^{\emph{\cd}_{\mathcal{S}}
                   (\fa, X)}_{\fa}(X)$ with $T \in \mathcal{S}.$
                   In particular, $\emph{\H}^{\emph{\cd}_{\mathcal{S}}
                   (\fa, X)}_{\fa}(X)$ is not finite.}
                   \item[(ii)]{If $\emph{\cd}(\fa, X)> 0$,
                   then $\emph{\H}^{\emph{\cd}(\fa, X)}_{\fa}(X)/T$
                   is not coatomic for any proper submodule $T$ of $\emph{\H}^{\emph{\cd}(\fa, X)}_{\fa}(X)$.
                   In particular, $\emph{\H}^{\emph{\cd}(\fa, X)}_{\fa}(X)$ is not
                   coatomic.}
                   \item[(iii)]{If $\emph{\q}_{\fa}(X)> 0$, then $\emph{\H}_{\fa}^{\emph{\q}_{\fa}(X)}(X)/T$
                   is not minimax for any Artinian submodule $T$ of $\emph{\H}_{\fa}^{\emph{\q}_{\fa}(X)}(X)$.
                   In particular, $\emph{\H}_{\fa}^{\emph{\q}_{\fa}(X)}(X)$ is not minimax.}
               \end{itemize}
\end{cor}

\begin{proof} (i) Assume contrarily that $\H^{\tiny\cd_{\mathcal{S}}(\fa, X)}_{\fa}(X)/T$ is
finite. Then there exists an integer $j$ such that $\fa
^{j}(\H^{\tiny\cd_{\mathcal{S}}(\fa, X)}_{\fa}(X)/T)= 0$; that is
$\fa^j\H^{\tiny\cd_{\mathcal{S}}(\fa, X)}_{\fa}(X)\subseteq T$.
On the other hand, by Corollary \ref{3-2}, $\H^{\tiny\cd_{\mathcal{S}}(\fa,
X)}_{\fa}(X)/\fa ^{j}\H^{\tiny\cd_{\mathcal{S}}(\fa, X)}_{\fa}(X)$
is in $\mathcal{S}$ and so its quotient $\H^{\tiny\cd_{\mathcal{S}}(\fa,
X)}_{\fa}(X)/T$ is in $\mathcal{S}$. Therefore
$\H^{\tiny\cd_{\mathcal{S}}(\fa, X)}_{\fa}(X)$ is in $\mathcal{S}$
which contradicts the definition of $\cd_{\mathcal{S}}(\fa, X).$

(ii) Assume that $\H^{\tiny\cd(\fa, X)}_{\fa}(X)/T$ is coatomic.
There exists a maximal submodule $T'/T$ of $\H^{\tiny\cd(\fa, X)}_{\fa}(X)/T$
so that there is an exact sequence
$$0\lo T'/T\lo \H^{\tiny\cd(\fa, X)}_{\fa}(X)/T\lo R/\fm\lo 0$$
for some maximal ideal $\fm$ of $R$, which results the exact
sequence
$$T'/\fa T'+ T\lo \H^{\tiny\cd(\fa, X)}_{\fa}(X)/\fa \H^{\tiny\cd(\fa, X)}_{\fa}(X)+T\lo R/\fm\lo 0$$
if one applies the functor $R/\fa\otimes_R -$. It can be seen either
directly or deduced from Corollary \ref{3-2} that $\H^{\tiny\cd(\fa,
X)}_{\fa}(X)/\fa \H^{\tiny\cd(\fa, X)}_{\fa}(X)= 0$. Therefore its
homomorphic image $\H^{\tiny\cd(\fa, X)}_{\fa}(X)/\fa
\H^{\tiny\cd(\fa, X)}_{\fa}(X)+ T$ is zero. This contradiction shows
that $\H^{\tiny\cd(\fa, X)}_{\fa}(X)/T$ is not coatomic.

(iii) Assume, in contrary, that $\H_{\fa}^{\tiny\q_{\fa}(X)}(X)/T$
is a minimax module; so that there exists an exact sequence
\begin{equation}\label{3.1}
0\lo T'/T\lo \H_{\fa}^{\tiny\q_\fa(X)}(X)/T\lo
\H_{\fa}^{\tiny\q_{\fa}(X)}(X)/T'\lo 0
\end{equation}
such that $T'/T$ is finite and $\H_{\fa}^{\tiny\q_{\fa}(X)}(X)/T'$
is Artinian. There is an integer $j$ such that $\fa^{j}(T'/T)= 0$.
As, by Corollary \ref {3-2}, $\H_{\fa}^{\tiny\q_{\fa}(X)}(X)/\fa^{j}
\H_{\fa}^{\tiny\q_{\fa}(X)}(X)$ is Artinian its quotient
$\H_{\fa}^{\tiny\q_{\fa}(X)}(X)/\fa^{j}
\H_{\fa}^{\tiny\q_{\fa}(X)}(X)+ T$ is also Artinian. Applying the
functor $R/\fa^j\otimes_R -$ to the exact sequence (\ref {3.1})
yields the exact sequence
$$\Tor^R_1(R/\fa^j, \H_{\fa}^{\tiny\q_{\fa}(X)}(X)/T') \lo T'/T\lo
\H_{\fa}^{\tiny\q_{\fa}(X)}(X)/\fa^j \H_{\fa}^{\tiny\q_{\fa}(X)}(X)+
T$$ from which we obtain that $T'/T$ is Artinian. Now, (\ref {3.1})
implies that $\H_{\fa}^{\tiny\q_{\fa}(X)}(X)/T$ is Artinian which
contradicts with the fact that $\H_{\fa}^{\tiny\q_{\fa}(X)}(X)$ is
not Artinian.
\end{proof}



In \cite [Proposition 3.1] {Y}, it is proved that, for a positive integer $n$, $\H_\fa^i(X)= 0$ for all $i\geq n$
whenever $X$ and all modules $\H_\fa^i(X)$, for all $i\geq n$, are finite and the ground ring $R$ is local.
In the following, among other things, we generalize this result for a general ring $R$ and an arbitrary $R$--module $X$.

\begin{cor} \label {3-4} Let $X$ be an arbitrary $R$--module and let
$n$ be a positive integer. Then the following statements are
equivalent.
               \begin{itemize}
                   \item[(i)]{$\emph{\H}^{i}_{\fa}(X)$ is coatomic for all $i\geq n$.}
                   \item[(ii)]{$\emph{\H}^{i}_{\fa}(X)$ is finite for all $i\geq n$.}
                   \item[(iii)]{$\emph{\H}^{i}_{\fa}(X)= 0$ for all $i\geq n$.}
               \end{itemize}
\end{cor}

\begin{proof}
(i) $\Leftrightarrow$ (iii). This is clear from Corollary \ref
{3-3}(ii).

(ii) $\Leftrightarrow$ (iii). It follows from Corollary \ref
{3-3}(i).
\end{proof}



In consistence of Corollary \ref{3-4}, one can state the following
result about Artinian-ness of local cohomology modules from a point
upward.

\begin{cor} \label {3-5} Let $X$ be an arbitrary $R$--module and let
$n$ be a positive integer. Then the following statements are
equivalent.
               \begin{itemize}
                   \item[(i)]{$\emph{\H}^{i}_{\fa}(X)$ is minimax for all $i\geq n$.}
                   \item[(ii)]{$\emph{\H}^{i}_{\fa}(X)$ is Artinian for all $i\geq n$.}
               \end{itemize}
\end{cor}

\begin{proof}
This follows from Corollary \ref {3-3}(iii).
\end{proof}
{\it Acknowledgement.} The authors would like to thank the referee
for her/his comments.

\bibliographystyle{amsplain}

\end{document}